\documentclass[11pt, a4paper]{amsart}

\usepackage{amssymb, array, amsmath, amscd, pdfpages, enumerate, amsthm, fixltx2e, setspace, hyperref}

\DeclareMathOperator*{\Ran}{Ran}
\DeclareMathOperator*{\Ker}{Ker}
\DeclareMathOperator*{\Fix}{Fix}
\DeclareMathOperator*{\dist}{dist}
\DeclareMathOperator*{\R}{Re}

\newcommand{\e}{\mathrm{e}}
\newcommand{\ii}{\mathrm{i}}
\newcommand{\dd}{\mathrm{d}}
\newcommand{\B}{\mathcal{B}}
\newcommand{\LL}{\mathcal{L}}
\newcommand{\C}{\mathcal{C}}
\newcommand{\D}{\mathcal{D}}
\newcommand{\T}{\mathbb{T}}
\newcommand{\RR}{\mathbb{R}}
\newcommand{\CC}{\mathbb{C}}
\newcommand{\NN}{\mathbb{N}}
\newcommand{\ZZ}{\mathbb{Z}}
\newcommand{\DD}{\mathbb{D}}
\newcommand{\A}{\mathbb{A}}
\newcommand{\PP}{\mathbb{P}}

\newtheorem{thm}{Theorem}[section]

\theoremstyle{definition}
\newtheorem{ex}[thm]{Example}

\newtheorem{rem}[thm]{Remark}
\newtheoremstyle{mystyle}{}{}{\sl}{}{\bf}{.}{ }{}
\theoremstyle{mystyle}
\newtheorem{mythm}[thm]{Theorem}
\newtheorem{mylem}[thm]{Lemma}
\newtheorem{myprp}[thm]{Proposition}
\newtheorem{mycor}[thm]{Corollary}

\numberwithin{equation}{section}  

\begin{document}

\title[Rates of decay in the Katznelson-Tzafriri theorem]{Rates of decay in the classical Katznelson-Tzafriri theorem}
\author{David Seifert}
\address{Mathematical Institute, University of Oxford, Andrew Wiles Building, Radcliffe Observatory Quarter, Woodstock Road, Oxford\;\;OX2 6GG, United Kingdom}
\curraddr{Balliol College, Oxford\;\;OX1 3BJ, United Kingdom}
\email{david.seifert@balliol.ox.ac.uk}
\date{6 March 2014}

\begin{abstract}
 Given a power-bounded operator $T$, the theorem of Katznelson and Tzafriri states that $\|T^n(I-T)\|\to0$ as $n\to\infty$ if and only if the spectrum $\sigma(T)$ of $T$ intersects the unit circle $\T$ in at most the point $1$. This paper investigates the rate at which  decay takes place when $\sigma(T)\cap\T=\{1\}$. The results obtained lead in particular to both upper and lower bounds on this rate of decay in terms of the growth of the resolvent operator $R(\e^{\ii\theta},T)$ as $\theta\to0$.  In the special case of polynomial resolvent growth, these bounds are then shown to be optimal for general Banach spaces but not in the Hilbert space case. 
  
 \end{abstract}

\maketitle

\section{Introduction}
The Katznelson-Tzafriri theorem (see \cite[Theorem~1]{KT}) is one of the cornerstones of the asymptotic theory of operator semigroups; for surveys, see for instance \cite{Ba} and \cite{CT}. In its original and simplest form, the result concerns the asymptotic behaviour of $\|T^n(I-T)\|$ as $n\to\infty$ for suitable operators $T$ and has applications both in the theory of iterative methods (see \cite{Nev}) and to zero-two laws for stochastic processes (see \cite{KT} and \cite{OS}). Writing $\T$ for the unit circle $\{\lambda\in\CC:|\lambda|=1\}$, it can be stated as follows.

\begin{mythm}\label{KT_classical}
Let $X$ be a complex Banach space and let $T\in \B(X)$ be a power-bounded operator. Then \begin{equation}\label{KT}\lim_{n\to\infty}\|T^n(I-T)\|=0\end{equation}  if and only if $\sigma(T)\cap\T\subset\{1\}$.
\end{mythm}

Since its discovery in 1986, the Katznelson-Tzafriri theorem has attracted a considerable amount of interest, and this has lead to a number of extensions and improvements of the original result; see \cite[Section~4]{CT} for an overview, and also \cite{Leka}, \cite{Seifert} and \cite{Zarr}.  One aspect which so far has been studied only in special cases, however, is the \textsl{rate} at which decay takes place in \eqref{KT}; see for instance \cite{DLM}, \cite{Dun}, \cite[Chapter~4]{Nev},  \cite{Nev3} and \cite{Nev2}. Of course, for operators $T$ satisfying $\sigma(T)\cap\T=\emptyset$ the question is of no real interest, since in this case $r(T)<1$ and the decay is necessarily exponential. The focus here, therefore,  will be on the case where $\sigma(T)\cap\T=\{1\}$, with the aim of relating the rate of decay in \eqref{KT} to the growth of  the norm $\|R(\e^{\ii\theta},T)\|$ of the resolvent operator as $\theta\to0$. Once the behaviour of the resolvent near its singularity is adequately taken into account, it turns out to be possible not only to obtain explicit bounds on the quantity $\|T^n(I-T)\|$ for sufficiently large $n\geq0$ but also to establish their sharpness, or lack thereof, in an important special case.  Crucial in this undertaking are certain techniques which can be viewed as discrete analogues of ideas developed recently in the  context of $C_0$-semigroups, where they can be used to study energy decay for damped wave equations; see \cite{BEPS}, \cite{BCT}, \cite{BD}, \cite{BT}, \cite{Mar} and the references therein. 
 
The remainder of the paper divides into two  parts.  The first, Section~\ref{sec:main_rates}, contains the main general results. Most importantly, these include both a lower (Corollary~\ref{cor:lb}) and an upper bound (Theorem~\ref{thm}) in terms of the growth of the resolvent near $1$ for the quantity $\|T^n(I-T)\|$ when $T$ is a suitable power-bounded operator and $n\geq0$ is sufficiently large.  Section~\ref{sec:optimality_poly} then  investigates the optimality of these bounds in the case of polynomial resolvent growth. The two main results here show, respectively, that in this situation no tighter bounds may be found for operators on general Banach spaces (Theorem~\ref{log_optimality}) but that a stronger conclusion holds if the underlying space is assumed to be a Hilbert space (Theorem~\ref{hilbert_poly_decay}).

The notation used  throughout is as follows. Given a complex Banach space $X$, let $\B(X)$ stand for the algebra of bounded linear operators on $X$. An operator $T\in\B(X)$  is said to be power-bounded if $\sup\{\|T^n\|:n\geq 0\}<\infty$. Denote the range and kernel  of an operator $T\in\B(X)$ by  $\Ran(T)$ and $\Ker(T)$, respectively, and write $\Fix(T):=\Ker(I-T)$ for the set of fixed points of $T$, $\sigma(T)$ for its spectrum and $r(T)$  for its spectral radius. Furthermore, given an element $\lambda$ of the resolvent set  $\rho(T):=\CC\backslash\sigma(T)$, let $R(\lambda,T):=(\lambda-T)^{-1}$ denote the resolvent operator of $T$. All remaining pieces of notation will be introduced as the need arises.

\section{General results}\label{sec:main_rates}

Let $T\in\B(X)$ be a power-bounded operator on a complex Banach space $X$, and suppose that $\sigma(T)\cap\T=\{1\}$. In order to address the question of rates of decay in \eqref{KT}, it will be convenient to have in place a few non-standard pieces of notation. Thus, given an operator $T$ as above, a decreasing function $m:(0,\pi]\to(0,\infty)$   such that  $\|R(\e^{\ii \theta},T)\|\leq m(|\theta|)$ for all $\theta$ with $0<|\theta|\leq\pi$ will be said to be a \textsl{dominating function (for the resolvent of $T$)}. Likewise a decreasing function  $\omega:\ZZ_+\to(0,\infty)$  such that $\|T^n(I-T)\|\leq\omega(n)$ for all $n\in\ZZ_+$ will be said to be a \textsl{dominating function (for  $T$)}. The \textsl{minimal} dominating functions are given, for $\theta\in(0,\pi]$ and $n\geq0$, by 
\begin{equation}\label{min_df}\begin{aligned}
m(\theta)&=\sup\big\{\|R(\e^{\ii\vartheta},T)\|:\theta\leq|\vartheta|\leq\pi\big\},\\  
\omega(n)&=\sup\big\{\|T^k(I-T)\|:k\geq n\big\},
\end{aligned}\end{equation} 
respectively. Thus, for the minimal dominating function $\omega$ of $T$,  $\omega(n)\to0$ as $n\to\infty$  precisely when \eqref{KT} holds. Note also that the function $m$ defined in \eqref{min_df} is continuous. In what follows, the same will be assumed to be true of any dominating function $m$ for the resolvent of $T$. In particular, any such dominating function $m$ possesses a right-inverse $m^{-1}$ defined on the range of $m$. On the other hand, given a dominating function $\omega$ for $T$ which satisfies $\omega(n)\to0$ as $n\to\infty$, define the function $\omega^*:(0,\infty)\to\ZZ_+$ by \begin{equation}\label{w_inv}\omega^*(s):=\min\big\{n\in\ZZ_+:\omega(n)\leq s\big\}.\end{equation} Then $\omega(\omega^*(s))\leq s$ for all $s>0$, with equality for all $s$ in the range of $\omega$.

 Recall the elementary estimate \begin{equation}\label{resolvent_lb}\|R(\lambda,T)\|\geq\frac{1}{\dist(\lambda, \sigma(T))},\end{equation} which holds for all $\lambda\in\rho(T)$. Since $1\in\sigma(T)$, it follows that $m(\theta)\geq  \theta^{-1}$ for all $\theta\in(0,\pi]$. Thus there is a minimal rate at which the resolvent of any  operator $T$  as above must blow up near its singularity. This may suggest that there should exist a corresponding minimal rate, independent of the operator $T$, at which decay  takes place in \eqref{KT}. As Corollary~\ref{cor:lb} below will show, however, this is far from being the case; see also  \cite[Theorem~4.2]{AR}.  The next result, on the other hand,  shows that instances in which the decay is faster than that of $n^{-1}$ are of a very special nature. It is a direct analogue of \cite[Theorem~6.7]{BCT}; see also \cite[Remarks~2.3 and~2.4]{Nev3}.

 \begin{mythm}\label{max_decay}
 Let $X$ be a complex Banach space and let $T\in\B(X)$ be a power-bounded operator such that $\sigma(T)\cap\T=\{1\}$. Then either \begin{equation}\label{limsup_ritt}
 \limsup_{n\to\infty}n\|T^n(I-T)\|>0
 \end{equation} or there exist closed $T$-invariant subspaces $X_0$ and $X_1$ of $X$ such that   $X_0\subset\Fix(T)$, the restriction $T_1$ of $T$ to $X_1$ satisfies  $r(T_1)<1$ and $X=X_0\oplus X_1$.
 \end{mythm}
 
 \begin{proof}[\textsc{Proof}] Supposing first that 1 is a limit point of $\sigma(T)$, let $\lambda_j\in\sigma(T)\backslash\{1\}$ be such that $\lambda_j\to1$ as $j\to\infty$ and set $n_j:=\lfloor|1-\lambda_j|^{-1}\rfloor$. Since $r(T^n(I-T))\leq\|T^n(I-T)\|$ for all $n\geq1$, it follows that $$ \limsup_{n\to\infty}n\|T^n(I-T)\|\geq \lim_{j\to\infty}\frac{n_j}{n_j+1}\left(1-\frac{1}{n_j}\right)^{n_j}=\e^{-1},$$and hence \eqref{limsup_ritt} holds.
 
If 1 is an isolated point of $\sigma(T)$, on the other hand, then  a standard spectral decomposition argument (see for instance \cite[Proposition~B.9]{Green_book}) shows that there exist closed $T$-invariant subspaces $X_0$ and $X_1$ of $X$ and a bounded projection $P$ of $X$ onto $X_1$ along $X_0$ which commutes with $T$.  In particular, $X=X_0\oplus X_1$. Moreover, the restrictions $T_0$ and $T_1$ of $T$ to $X_0$ and $X_1$ satisfy  $\sigma(T_0)=\{1\}$ and $\sigma(T_1)=\sigma(T)\backslash\{1\}$, respectively. Now, if \eqref{limsup_ritt} fails, then $$\liminf_{n\to\infty}n\|T_0^n(I-T_0)\|=0$$ and it follows from \cite[Theorem~2.2]{KMSOT} that $Tx=x$ for all $x\in X_0$, as required.
\end{proof}

\begin{rem}
It is easily seen that, if  $X$ splits, then in fact $X_0=\Fix(T)$  and $X_1=\Ran(I-T)$. In particular, $\Ran(I-T)$ is closed; see also \cite[Theorem~4.4.2]{Nev}.
\end{rem}

Thus $\|T^n(I-T)\|$ decays either at least exponentially as $n\to\infty$ or at a rate no faster than  $n^{-1}$.  The case of decay at this borderline rate turns out to be connected with a special class of operators. Recall that an operator $T$ on a complex Banach space $X$ is said to be a \textsl{Ritt operator} if $\sigma(T)\cap\T=\{1\}$ and there exists a constant $C>0$ such that \begin{equation}\label{Ritt}\|R(\lambda,T)\|\leq\frac{C}{|1-\lambda|}\end{equation} for $|\lambda|>1$; various interesting results on Ritt operators may be found for instance in \cite{ALM}, \cite{Blu2}, \cite{Blu},  \cite{Dun3}, \cite{Dun4}, \cite{Dun2}, \cite{EFR}, \cite{LLM}, \cite{LM}, \cite{LMX} and  \cite{Vit}. The reason why Ritt operators are  important in the present context is that a power-bounded operator $T$ satisfying $\sigma(T)\cap\T=\{1\}$ is a Ritt operator if and only if  $\|R(\e^{\ii\theta},T)\|=O(|\theta|^{-1})$ as $|\theta|\to0$; see for instance  the proof of Lemma~\ref{resolvent_lem} below. The following result shows that these operators are precisely those for which the rate of decay in \eqref{KT} is no slower than $n^{-1}$. This characterisation was obtained independently in \cite{Lyu2} and \cite{NZ}; see also \cite[Theorem~4.5.4]{Nev}.

\begin{mythm}\label{Ritt_thm}
Let $X$ be a complex Banach space. An operator $T\in\B(X)$ is a Ritt operator if and only if it is power-bounded and $\|T^n(I-T)\|=O(n^{-1})$ as $n\to\infty$.
\end{mythm}

Thus decay in \eqref{KT} at a rate no slower than that of $n^{-1}$ already implies a strong condition on the growth of the resolvent near its singularity at $1$.  The next result establishes a corresponding resolvent bound in a rather more general situation; see \cite[Theorem~6.10]{BCT} for an analogous result in the setting of $C_0$-semigroups.

\begin{mythm}\label{thm2}
Let $X$ be a complex Banach space and let  $T\in \B(X)$ be  a power-bounded operator. Suppose that $\omega$ is a dominating function for $T$ such that $\omega(n)\to0$ as $n\to\infty$, and let $\omega^*$ be as defined in \eqref{w_inv}. Then  $\sigma(T)\cap\T\subset \{1\}$ and, for any $c\in(0,1)$,  \begin{equation}\label{eq2}\|R(\e^{\ii\theta},T)\|=O\left( \frac{1}{|\theta|}+\omega^*\left(c|\theta|\right)\right)\end{equation} as $|\theta|\to0$.
 \end{mythm}
 
\begin{proof}[\textsc{Proof}] Suppose that $\lambda \in\sigma(T)\cap\T$. By the spectral mapping theorem for polynomials, $\lambda^n(1-\lambda)\in\sigma(T^n(I-T))$  and hence $|1-\lambda|\leq\omega(n)$ for all $n\geq0$.  Letting $n\to\infty$, it follows that $\lambda=1$, so $\sigma(T)\cap\T\subset\{1\}$. 

Now let $\lambda\in \T\backslash\{1\}$. Then, for $n\geq0$, $$\begin{aligned}\lambda^n(1-\lambda) -T^n(I-T)&= (1-\lambda)\lambda^{n-1}\sum_{k=0}^{n-1}\lambda^{-k}T^k(\lambda-T)-T^n(\lambda-T)\end{aligned}$$ and hence, letting $M:=\sup\{\|T^n\|:n\geq 0\}$, $$|1-\lambda| \|R(\lambda,T)x\|\leq \omega(n)\|R(\lambda,T)x\|+M\big(1+n|1-\lambda|\big)\|x\|$$  for all  $x\in X$. Fix $b\in (c,1)$  and  let  $n=\omega^*(b|1-\lambda|)$. Then $$\|R(\lambda,T)\|\leq \frac{M}{1-b}\left(\frac{1}{|1-\lambda|}+\omega^*\left(b|1-\lambda|\right)\right)$$ and,  since $b|1-\lambda|\geq c|\theta|$ whenever $\lambda=\e^{\ii\theta}$ for some sufficiently small $\theta\in(-\pi,\pi]\backslash\{0\}$, the result follows.
\end{proof}

\begin{rem}\label{lb_rem}
A similar argument shows that, given any constant $K>M$, where $M$ is as above, there exists $c\in(0,1)$ such that \begin{equation}\label{eq}\|R(\e^{\ii\theta},T)\|\leq K\left( \frac{1}{|\theta|}+\omega^*\left(c|\theta|\right)\right)\end{equation} whenever $|\theta|$ is sufficiently small.  Note also that, by \eqref{resolvent_lb}, the $|\theta|^{-1}$ term in \eqref{eq2} and \eqref{eq} cannot in general be omitted.  
\end{rem}

In analogy with \cite[Corollary~6.11]{BCT}, these observations can be used to obtain a lower bound on the quantity $\|T^n(I-T)\|$ when $n\geq0$ is large.

\begin{mycor}\label{cor:lb}
Let $X$ be a complex Banach space, let $T\in\B(X)$ be a power-bounded operator such that  $\sigma(T)\cap\T=\{1\}$ and  let  $m$ be the minimal dominating function for the resolvent of $T$ defined in \eqref{min_df}. Suppose that \begin{equation}\label{L}\lim_{\theta\to0}\max\big\{\|\theta R(\e^{\ii \theta},T)\|,\|\theta R(\e^{-\ii \theta},T)\|\big\}=\infty.\end{equation} Then, given any right-inverse $m^{-1}$ of $m$, there exist constants $c,C>0$  such that \begin{equation}\label{KT_lb}\|T^n(I-T)\|\geq c m^{-1}(Cn)\end{equation} for all sufficiently large $n\geq0$.
\end{mycor}

\begin{proof}[\textsc{Proof}] 
Let $\omega$ be as defined in \eqref{min_df}.  Since $\omega(n)\to0$ as $n\to\infty$ by Theorem~\ref{KT_classical}, it follows from Theorem~\ref{thm2} that there exists $B>0$  such that $m(\theta)\leq  B\left( \theta^{-1}+\omega^*\left(\theta/2\right)\right)$ for all sufficiently small $\theta\in(0,\pi]$, and hence     \begin{equation}\label{w_m}\omega^*(\theta/2)\geq m(\theta)\left(\frac{1}{B}-\frac{1}{\theta m(\theta)}\right)\end{equation} for all such values of $\theta$. Let $C:=2B$ and, for $n\geq0$,  let $\theta_n:=2\omega(n)$. By  \eqref{L}, $\theta_n m(\theta_n)>C$ for all sufficiently large  $n\geq0$, so  \eqref{w_m} implies that $\omega^*(\theta_n/2)> C^{-1}m(\theta_n)$ for each such $n\geq0$. Since $\omega^*(\theta_n/2)\leq n$ and therefore $$m(m^{-1}(Cn))=Cn\geq C\omega^*(\theta_n/2)>m(\theta_n),$$ it follows that  $m^{-1}(Cn)<\theta_n$ for all sufficiently large $n\geq0$. Moreover, $\theta_n\leq2 M\|T^n(I-T)\|$ for all $n\geq0$, where $M:=\sup\{\|T^n\|:n\geq 0\}$, which shows that  \eqref{KT_lb} holds for $c=(2M)^{-1}$.
\end{proof}

\begin{rem}
A similar argument using Remark~\ref{lb_rem} instead of Theorem~\ref{thm2} shows that the conclusion \eqref{KT_lb} remains true if \eqref{L} is replaced by the weaker condition that $L>M$, where $M$ is as  above  and $$L:=\liminf_{\theta\to0}\max\big\{\|\theta R(\e^{\ii \theta},T)\|,\|\theta R(\e^{-\ii \theta},T)\|\big\}.$$ Taking $T$ to be the identity operator shows that the conclusion can be false when $L=M$.
\end{rem}

Suppose that $T$ is a power-bounded operator such that $\sigma(T)\cap\T=\{1\}$ and let $m$ be the minimal dominating function for the resolvent of $T$ defined in \eqref{min_df}. If $T$ is  a Ritt operator, then it follows from Theorem~\ref{Ritt_thm} that, for any $c\in(0,1)$,  $$\|T^n(I-T)\|=O\big(m^{-1}(cn)\big)$$ as $n\to\infty$ and, in view of Corollary~\ref{cor:lb}, this type of upper bound is in general the best one can hope for. The next result describes the class of functions $m$ for which such an upper bound is satisfied in the case of a normal operator on a Hilbert space; see also \cite[Proposition~6.13]{BCT}.

\begin{myprp}\label{normal_thm}
Let $X$ be a complex Hilbert space and let $T\in\B(X)$ be a power-bounded normal operator such that $\sigma(T)\cap\T=\{1\}.$
Furthermore, let $m$ be the minimal dominating function for the resolvent of $T$ defined in \eqref{min_df}, let $m^{-1}$ be any right-inverse of $m$ and let  $S\subset\NN$. 
\begin{enumerate}
\item Suppose there exist constants $c,C>0$ such that 
\begin{equation}\label{m_decay}\|T^n(I-T)\|\leq C m^{-1}(c n)\end{equation} for all $n\in S$. Then, for any $b\in (0,c)$, there exists a constant $B>0$ such that  \begin{equation}\label{log_m}\frac{m(\theta)}{m(\vartheta)}\geq b \log\frac{\vartheta}{\theta}-B,\end{equation} for all $\theta\in (0,\pi]$ of the form $\theta=m^{-1}(cn)$ with $n\in S$ and for all sufficiently small $\vartheta\in (0,\pi]$ . 
\item Conversely, if there exist constants $b,B>0$ such that \eqref{log_m} holds for all $\theta\in (0,\pi]$ of the form $\theta=m^{-1}(bn)$ with $n\in S$ and all  $\vartheta\in(0,\pi]$, then there exists a constant $C>0$ such that \eqref{m_decay} holds with $c=b$.
\end{enumerate}
 \end{myprp}
 
 \begin{proof}[\textsc{Proof}] 
Note first that, for $\theta\in(0,\pi]$, $$m(\theta)^{-1}=\min\big\{|\lambda-\e^{\ii\varphi}|:\lambda\in\sigma(T), \theta\leq|\varphi|\leq\pi\big\},$$ and that \eqref{m_decay} is equivalent to having 
 \begin{equation}\label{equiv_dec} n\log\frac{1}{|\lambda|}\geq\log\frac{|1-\lambda|}{C m^{-1}(cn)}\end{equation}
for all $\lambda\in\sigma(T)\backslash\{1\}$ and all  $n\in S$.

Suppose this holds and let  $\theta=m^{-1}(cn)$ for some  $n\in S$. Then $$m(\theta)\geq \frac{c}{\log\frac{1}{|\lambda|}}\log\frac{|1-\lambda|}{C\theta}$$ for all $\lambda\in\sigma(T)\backslash\{1\}$.  Define the function $g:(0,1)\to\RR$ by $g(s):=\frac{s-1}{\log s}.$  Then $g$  is a  continuous increasing function satisfying $g(s)\to1$ as $s\to1$, and in fact  $$g(s)=\inf\left\{\frac{r-1}{\log r}:s< r<1\right\}$$ for all $s\in(0,1)$. Thus, given any  $b\in(0,c)$, there exists $s_0\in(0,1)$ such that $cg(s_0)>b$. Now suppose that $\vartheta\in(0,1-s_0)$, let $\lambda\in\sigma(T)$ be such that $m(\vartheta)=|\lambda-\e^{\ii\varphi}|^{-1}$ for some $\varphi\in(0,\pi]$ with $|\varphi|\geq\vartheta$, and let  $r:=|\lambda|$. Since  $m(\vartheta)\geq\vartheta^{-1}$,  it follows from the estimate $1-r\leq|\lambda-\e^{\ii\varphi}|$ that $r> s_0$. Thus, if $|1-\lambda|\geq\frac{\vartheta}{2}$, then $$\frac{m(\theta)}{m(\vartheta)}\geq \frac{c|\lambda-\e^{\ii\varphi}|}{\log\frac{1}{r}}\log\frac{|1-\lambda|}{C\theta}\geq b\log\frac{\vartheta}{2C\theta},$$ which gives \eqref{log_m} with $B=b\log 2C.$ If $|1-\lambda|<\frac{\vartheta}{2}$, on the other hand, then $|\lambda-\e^{\ii\varphi}|\geq\frac{\vartheta}{3}$ and hence $$\frac{m(\theta)}{m(\vartheta)}\geq \frac{\vartheta}{3\theta}\geq b\log\left(\frac{\theta}{3b\vartheta}\right) ,$$ which gives \eqref{log_m} with  $B=b\log 3b$. Thus, taking $B=b\max\{\log 2C,\log 3b\}$,  the proof the first statement is complete.

Now suppose, conversely, that \eqref{log_m} holds for some constants $b,B>0$,  all $\theta\in (0,\pi]$ of the form $\theta=m^{-1}(cn)$ with $n\in S$ and all  $\vartheta\in(0,\pi]$. Let $\lambda=r\e^{\ii\phi}\in\sigma(T)\backslash\{1\}$, and set $\vartheta:=|\phi|$, so that  $$\log\frac{1}{r}\geq 1-r=|\e^{\ii\phi}-\lambda|\geq\frac{1}{m(\vartheta)}.$$ Hence, if $\vartheta\geq\frac{1}{2}|1-\lambda|$, then  \eqref{log_m} gives $$n\log\frac{1}{r}\geq\frac{1}{b}\frac{m(m^{-1}(bn))}{m(\vartheta)}\geq\log\left(\frac{\vartheta}{m^{-1}(bn)}\right)-\frac{B}{b}\geq\log\left(\frac{|1-\lambda|}{2m^{-1}(bn)}\right)-\frac{B}{b},$$  thus establishing \eqref{equiv_dec} with $c=b$ and $C=2\e^{B/b}$. On the other hand, if $\vartheta<\frac{1}{2}|1-\lambda|$, then $1-r\geq\frac{1}{2}|1-\lambda|$ and consequently $$n\log\frac{1}{r}\geq n(1-r)\geq\frac{|1-\lambda|}{2b}m(m^{-1}(bn))\geq \frac{|1-\lambda|}{2bm^{-1}(bn)}\geq\log\left(\frac{|1-\lambda|}{2bm^{-1}(bn)}\right),$$ which gives \eqref{equiv_dec} with $c=b$ and $C=2b$. Thus taking $C=2\max\{\e^{B/b},b\}$ finishes  the proof.
 \end{proof}
 
 \begin{rem}\label{normal_rem}
 The result remains true, with the same proof, for any complex Banach space $X$ and any power-bounded operator $T\in\B(X)$ satisfying  $$\| f(T)\|=\sup\big\{|f(\lambda)|:\lambda\in\sigma(T)\big\}$$ for all functions $f$ of the form $f(\lambda)=\lambda^n(1-\lambda)$ with $n\geq0$ or $f(\lambda)=(\mu-\lambda)^{-1}$ with $\mu\in\rho(T)$. This includes, in particular, the class of multiplication operators on any of the classical function or sequence spaces. Note also that the second of the two implications holds more generally when $m$ is an arbitrary dominating function for the resolvent of $T$.
 \end{rem}
 
 Thus \eqref{m_decay} holds for a normal operator $T$ if and only if the minimal dominating function $m(\theta)$ for the resolvent of $T$ grows in a fairly regular way as  $\theta\to0$. The following example exhibits a class of normal operators for which this is not the case.
 
\begin{ex}
Let $X=\ell^2$. Given a strictly increasing sequence $(r_k)$ of positive terms such that $r_k\to1$ as $k\to\infty$, let  $\lambda_k:=r_k\e^{\ii/k}$ and  consider the operator $T\in\B(X)$ given by $Tx:=(\lambda_k x_k)$.   Then $T$ is a normal contraction with $\sigma(T)=\{\lambda_k:k\geq1\}\cup\{1\}$. Moreover, if $r_k>1-k^{-2}$ for all $k\geq1$, then $1-r_k<|\e^{\ii/k}-\lambda_j|$ whenever $j\ne k$ and hence $$m(k^{-1})=\|R(\e^{\ii/k},T)\|=\frac{1}{1-r_k}$$ for all $k\geq1$, where $m$ is the minimal dominating function for the resolvent of $T$ defined in \eqref{min_df}. Suppose moreover that $\log r_{k!+1}\geq2\log r_{(k+1)!}$ for all $k\geq1$ and, given  $c>0$, let $b\in(0,c)$ and $n_k:=\lceil-(b\log r_{(k+1)!})^{-1}\rceil$. Then  $$m\left(((k+1)!)^{-1}\right)=\frac{1}{1-r_{(k+1)!}}\sim-\frac{1}{\log r_{(k+1)!}}\sim bn_k$$  as $k\to\infty$, and hence $m^{-1}(cn_k)\leq ((k+1)!)^{-1}$ for all sufficiently large $k\geq1$. Since $|1-\lambda_{k!+1}|\geq (3k!)^{-1}$ and $|\lambda_{k!+1}^{n_k}|\geq\e^{-2/b}$ for all $k\geq1$, it follows that  $$\|T^{n_k}(I-T)\|\geq |\lambda_{k!+1}^{n_k}(1-\lambda_{k!+1})|\geq \frac{1}{3\e^{2/b}k!}\geq \frac{k}{3e^{2/b}} m^{-1}(cn_k)$$ when  $k\geq1$ is sufficiently large. In particular, \eqref{m_decay} fails to hold for every $c>0$. For an analogous example in the continuous-time setting, see \cite[Example~4.4.15]{Green_book}.
\end{ex}

Thus, given a power-bounded operator $T$ such that $\sigma(T)\cap\T=\{1\}$, a right-inverse $m^{-1}$ of some  dominating function $m$ for the resolvent of $T$ and a constant $c\in (0,1)$, it is not in general the case that $\|T^n(I-T)\|=O(m^{-1}(cn))$ as $n\to\infty$. The next result shows that it is nevertheless possible to obtain  an upper bound of this kind provided the function $m$ is  modified appropriately.  Indeed, given an operator $T$ as above and a dominating function $m$ for the resolvent of $T$, define the function $ m_{\log}:(0,\pi]\to(0,\infty)$ by 
\begin{equation}\label{m_log}
 m_{\log}(\theta):=m(\theta)\log\left(1+\frac{m(\theta)}{\theta}\right),
\end{equation}
noting that this function is strictly decreasing and hence possesses a well-defined inverse $ m_{\log}^{-1}$ defined on the range of $m_{\log}$. As Theorem~\ref{thm} below shows, the above upper bound on $\| T^n(I-T)\|$ for large values of $n\geq0$ is valid when  $m^{-1}$ is replaced by  $m_{\log}^{-1}$. This raises the question by how much the asymptotic behaviour of these two functions differs in particular instances. If $m(\theta)=C\e^{\alpha/\theta}$, for example,  where $C,\alpha>0$ are constants, then $ m_{\log}^{-1}(s)\sim \frac{\alpha}{\log s}$ as $s\to\infty$, so $ m_{\log}^{-1}$  has the same asymptotic behaviour as $ m^{-1}$ in this case. On the other hand, if $m(\theta)=C\theta^{-\alpha}$ for some  constants $C>0$ and $\alpha\geq1$,  then $ m_{\log}^{-1}(s)\sim (\frac{\log s}{s})^{1/\alpha}$ as $s\to\infty$, so $ m_{\log}^{-1}$ differs from $ m^{-1}$ by a logarithmic factor. For similar examples in the continuous-time setting, see \cite[Example~1.4]{BD} and \cite[Section~2]{Mar}.

Throughout the proof of the next result, and also in various other places later on, the letters $c$ and $C$, if used without having been introduced explicitly, stand for  positive constants, which will be thought of as being small and large, respectively, and which need not be the same at each occurrence. The result itself is a discrete analogue of \cite[Proposition~3.1]{Mar}, which in turn is a development of \cite[Theorem~1.5]{BD}; see also \cite[Chapter~VI]{FS}, where similar techniques are discussed in the context of combinatorial problems.

\begin{mythm}\label{thm}
Let $X$ be a complex Banach space and let $T\in \B(X)$ be a power-bounded operator such that $\sigma(T)\cap\T=\{1\}$. Furthermore, let $m$ be a dominating function for the resolvent of $T$ and let $ m_{\log}$ be as defined in \eqref{m_log}. Then, for any  $c\in(0,1)$,  $$\|T^n(I-T)\|=O\big(  m_{\log}^{-1}(c n)\big)$$ as $n\to\infty$.
 \end{mythm}

\begin{proof}[\textsc{Proof}]    Having fixed a dominating function $m$ and a constant $c\in(0,1)$, let $ \Omega$ denote the closure of the set $$\left\{r\e^{\ii\theta}\in\CC: 0\leq r\leq1-\frac{c}{m(|\theta|)}, 0<|\theta|\leq\pi\right\}.$$ Moreover, noting that $\sigma(T)\subset \Omega$ by a standard Neumann series argument, define the function $F_n:\CC\backslash\Omega\to\B(X)$  by $$F_n(\lambda):=T^n(2-T)\big(I-(\lambda-1)R(\lambda,T)\big).$$ It then follows from the resolvent identity that \begin{equation}\label{Fn}F_n(\lambda)=T^n(2-T)^2R(\lambda,T)\big(I-R(2,T)\big),\end{equation} and hence $F_n(2)=T^n(I-T)$. Thus, by Cauchy's integral formula,  $$T^n(I-T)=\frac{1}{2\pi\ii}\oint_\Gamma \frac{h(\lambda)}{\lambda-2}F_n(\lambda)\,\dd\lambda,$$ where $\Gamma$ is any contour outside $\Omega$ around the point $2$ and where  $h$ is any function that is holomorphic in the relevant region and satisfies $h(2)=1$. In what follows, it will be convenient to take $\Gamma=\Gamma_{\!\mathrm{in}}\cup\Gamma_{\!\mathrm{out}}$ to consist of an outer contour $\Gamma_{\!\mathrm{out}}$, which encloses both the point $2$ and the set $\Omega$, and an inner contour $\Gamma_{\!\mathrm{in}}$, which lies in the interior of $\Gamma_{\!\mathrm{out}}$ and incloses $\Omega$ but not the point $\lambda=2$.  Such a contour can be thought of as being closed by inserting a cut from any point on $\Gamma_{\!\mathrm{in}}$ to any point on $\Gamma_{\!\mathrm{out}}$, the contributions along which cancel out.

Let  $\varphi$ be the Cayley transform  defined by $\varphi(\lambda):=\frac{1-\lambda}{1+\lambda} $ and, for  $r\in(0,1)$ and $R>0,$  let \begin{equation}\label{gamma_r}\gamma_r:=\left\{\lambda\in\CC:\left|\lambda-\frac{1+r^2}{1-r^2}\right|=\frac{2r}{1-r^2}\right\}\end{equation}  and  $\Gamma_R:=\{\lambda\in\CC:|\lambda-1|=R\}$,  noting that $\varphi$ maps  $\gamma_r$ onto $r \T$, the real line onto itself and the unit circle $\T$ onto the imaginary axis. Now suppose that $r\in (0,\frac{1}{3})$ and $R>2$, and let  $\Gamma_{\!\mathrm{out}}=\Gamma_R$ and $\Gamma_{\!\mathrm{in}}=C_r\cup\gamma_r^+$, where $\gamma_r^+$ denotes the part of  $\gamma_r$  that lies outside the unit disc $\DD:=\{\lambda\in\CC:|\lambda|<1\}$  and where $C_r$ is any suitable path in $\DD\backslash\Omega$ connecting the endpoints of $\gamma_r^+$.  Furthermore, choose for $h$  the map $h_r$ given by $$h_r(\lambda):=\frac{1}{1+9r^2}\left(1+\frac{r^2}{\varphi(\lambda)^2}\right),$$ so that $h_r$ is holomorphic away from $1$.    Now, letting $M:=\sup\{\|T^n\|:n\geq0\}$, it follows from the series expansion of the resolvent that  \begin{equation}\label{eq:resolvent_bound}\|R(\lambda,T)\|\leq\frac{M}{|\lambda|-1}\end{equation} whenever $|\lambda|>1$ and hence, by \eqref{Fn} and the fact that $T$ is power-bounded,  $\|F_n\|\leq C(|\lambda|-1)^{-1}$ for  all such $\lambda$, where $C$ is independent of $n\geq0$. Since $h_r$ is bounded above in modulus independently of $r$ along $\Gamma_R$, it follows that  $$\left\|\oint_{\Gamma_R} \frac{h_r(\lambda)}{\lambda-2}F_n(\lambda)\,\dd\lambda\right\|\leq\frac{C}{R},$$ where $C$ is independent of $n\geq0$, and hence, by appealing to Cauchy's theorem and allowing $R\to\infty$, this contribution can be neglected.

Next note that, for $\lambda\in r\T$, the function $g_r$ defined by $g_r(\lambda):=1+r^2\lambda^{-2}$ satisfies $|g_r(\lambda)|=2r^{-1}|\R \lambda|$.  Moreover, an elementary calculation shows that, for $\lambda\in\gamma_r$, \begin{equation}\label{Re}1-|\lambda|^2=\frac{4\R\varphi(\lambda)}{1+2\R\varphi(\lambda)+r^2}.\end{equation} Since $h_r=\frac{g_r\circ\varphi}{1+9r^2}$ and $\varphi(\gamma_r)=r\T$, it follows that  \begin{equation}\label{h_bound}|h_r(\lambda)|\leq C\frac{|\R\varphi(\lambda)|}{r}\leq C \frac{||\lambda|-1|}{r}\end{equation} for  all  $\lambda\in \gamma_r$. But for each $\lambda\in \gamma_r$, $|\lambda|-1\leq Cr$ and $|1-\lambda|\leq Cr$ so, by \eqref{eq:resolvent_bound} and the definition of $F_n$,  $$\|F_n(\lambda)\|\leq C\|I-(\lambda-1)R(\lambda,T)\|\leq \frac{Cr}{|\lambda|-1}$$ for all $\lambda\in\gamma_r^+$. Hence $$\left\|\int_{\gamma_r^+} \frac{h_r(\lambda)}{\lambda-2}F_n(\lambda)\,\dd\lambda\right\|\leq Cr,$$ where $C$ is independent of $n\geq0$, and it remains to control only the contribution along $C_r$.

Let $\theta_r\in (0,\frac{\pi}{2})$ denote the argument of the point at which $\gamma_r$ meets $\T$ in the upper half-plane and define the curve ${C}^\circ_r$, for $\theta_r\leq|\theta|\leq\pi$, by $${C}^\circ_r(\theta):=\left(1-\frac{c}{m(|\theta|)}\right)\e^{\ii\theta}.$$   Furthermore, let $C_r^\pm$ denote the rays given, for $1-cm(\theta_r)^{-1}\leq s\leq1,$ by $C_r^\pm(s):=s\e^{\pm\ii\theta_r}$ and set $C_r={C}^\circ_r\cup C_r^+\cup C_r^-$. Defining $$p_n(\lambda):=\sum_{k=0}^{n-1}\lambda^{n-k-1}T^k$$ for $\lambda\in\CC\backslash\Omega$, it follows from  the resolvent identity, the relation $p_n(\lambda)=(\lambda^n-T^n)R(\lambda,T)$ and some elementary manipulations that  $$F_n(\lambda)=\frac{1}{\lambda-2}(2-T)^2\Big((\lambda-1)\big(\lambda^nR(\lambda,T)-p_n(\lambda)\big)-T^n R(2,T)\Big)$$for all   $\lambda\in\CC\backslash\Omega$ with $\lambda\ne2$;  see also \cite[Lemma~2.2]{Mar}. Hence Cauchy's theorem  gives $$\begin{aligned}R(2,T)^2\int_{C_r} \frac{h_r(\lambda)}{\lambda-2}F_n(\lambda)\,\dd\lambda&=\int_{C_r} \frac{h_r(\lambda)(\lambda-1)\lambda^n}{(\lambda-2)^2}R(\lambda,T)\,\dd\lambda \\&\quad-\int_{\gamma_r^-}\frac{h_r(\lambda)}{(\lambda-2)^2}\big((\lambda-1)p_n(\lambda)+T^nR(2,T)\big)\,\dd\lambda,\end{aligned}$$ where $\gamma_r^-:=\gamma_r\cap\DD$.  To estimate the first integral on the right-hand side, note first that, by a standard Neumann series argument,    $\|R(\lambda,T)\|\leq (1-c)^{-1}m(\theta_r)$ for all $\lambda\in {C}^\circ_r$. Since $h_r$ is uniformly bounded independently of $r$ along ${C}^\circ_r$, it follows that $$\left\|\int_{{C}^\circ_r}\frac{h_r(\lambda)(\lambda-1)\lambda^n}{(\lambda-2)^2}R(\lambda,T)\,\dd\lambda\right\|\leq C m(\theta_r)\left(1-\frac{c}{m(\theta_r)}\right)^n.$$ Similarly, $$\left\|\int_{C_r^\pm}\frac{h_r(\lambda)(\lambda-1)\lambda^n}{(\lambda-2)^2}R(\lambda,T)\,\dd\lambda\right\|\leq C\int_{1-cm(\theta_r)^{-1}}^1 s^n\,\dd s\leq\frac{C}{n+1}.$$ To bound the  integral along $\gamma_r^-$, note that $$\|(\lambda-1)p_n(\lambda)\|\leq \frac{Cr}{1-|\lambda|}$$ for all $\lambda\in\gamma_r^-$.  Thus by  \eqref{h_bound} both $h_r(\lambda)$ and $h_r(\lambda)(\lambda-1)p_n(\lambda)$ are uniformly bounded, independently of $r$ and $n$, as $\lambda$ ranges over $\gamma_r^-$, and it follows that
$$\left\|\int_{\gamma_r^-}\frac{h_r(\lambda)}{(\lambda-2)^2}\big((\lambda-1)p_n(\lambda)+T^nR(2,T)\big)\,\dd\lambda\right\|\leq Cr,$$ where $C$ is independent of $n\geq0$. 

Since $C^{-1}r\leq \theta_r\leq Cr$, combining  these bounds gives $$\|T^n(I-T)\|\leq C\left(\theta_r+\frac{1}{n+1}+m(\theta_r)\left(1-\frac{c}{m(\theta_r)}\right)^n\right),$$ where $C$ is independent of $n\geq0$ and $r\in(0,\frac{1}{3})$.  Now, if $n\geq0$ is sufficiently large, choosing $r\in(0,\frac{1}{3})$ so as to satisfy  $\theta_r= m_{\log}^{-1}(c n)$ gives $\exp(m(\theta_r)^{-1}cn)=1+\theta_r^{-1}m(\theta_r)$ and hence
$$m(\theta_r)\left(1-\frac{c}{m(\theta_r)}\right)^n\leq m(\theta_r)\exp\left(-\frac{c n}{m(\theta_r)}\right)\leq\theta_r.$$ Since moreover  $(n+1)^{-1}\leq C m_{\log}^{-1}(cn)$ for all $n\geq0$, this completes  the proof.
\end{proof}

\begin{rem}
As in \cite{Mar}, it is possible to obtain an analogous result when $\sigma(T)\cap\T$ is finite by replacing $I-T$ with a finite product of linear terms of the form $\e^{i\theta}-T$ with  $\theta\in(-\pi,\pi]$.
\end{rem}

\section{Optimality in the case of polynomial resolvent growth}\label{sec:optimality_poly}

  Suppose that $X$ is a complex Banach space and that $T\in\B(X)$ is a power-bounded operator such that $\sigma(T)\cap\T=\{1\}$. The purpose of this section is to investigate the optimality of Theorem~\ref{thm2} in the special case where the resolvent of $T$ grows at most polynomially, which is to say it admits a  dominating function of the form $m(\theta)=C\theta^{-\alpha}$ for some constants $C>0$ and $\alpha\geq1$, where the restriction on the parameter $\alpha$ is a consequence of \eqref{resolvent_lb}. Corollary~\ref{cor:lb}  and Theorem~\ref{thm}  combine to give the following result, which describes the range of decay rates that are possible in this situation. Here, given $\Omega\subset (0,\infty)$ and functions $f,g:\Omega\to(0,\infty)$, the notation $f(s)=\Theta(g(s))$ as $s\to0$ (or $s\to\infty$) means that there exist constants $c,C>0$ such that $cg(s)\leq f(s) \leq C g(s)$ for all sufficiently small (or large) values of $s\in\Omega$.

\begin{mycor}\label{cor:poly_bounds}
Let $X$ be a complex Banach space and let  $T\in \B(X)$ be a power-bounded operator such that $\sigma(T)\cap\T=\{1\}$. Suppose that, for some $\alpha\geq1$,  $\|R(\e^{\ii\theta},T)\|=\Theta(|\theta|^{-\alpha})$ as $\theta\to0$. Then there exist constants $c, C>0$ such that \begin{equation}\label{poly_bounds}\frac{c}{ n^{1/\alpha}}\leq\|T^n(I-T)\|\leq C\left(\frac{\log n}{ n}\right)^{1/\alpha}\end{equation}
 for all sufficiently large $n\geq0$.
\end{mycor}

The remainder of this section is concerned with the question whether the logarithmic factor on the right-hand side of \eqref{poly_bounds} is really needed. It follows from Proposition~\ref{normal_thm} and Remark~\ref{normal_rem} that it can be dropped whenever $T$ is a suitable multiplication operator on some function or sequence space.  The following example exhibits a less trivial case in which the same is true.

\begin{ex}\label{Toeplitz_example}
Let $X=\ell^p$ with  $1\leq p\leq\infty$ and, writing $S$ for the left-shift operator on $X$ given by $Sx:=(x_{k+1})$, define the operator $T\in\B(X)$ as $T:=\frac{1}{4}(I+S)^2$. Then $T$ is a (non-normal) Toeplitz operator of unit norm, with  $$\sigma(T)=\left\{r\e^{\ii\theta}\in\CC:-\pi<\theta\leq\pi\;\mbox{and}\;0\leq r\leq\frac{1+\cos\theta}{2}\right\}.$$ 
In particular, $\sigma(T)\cap\T=\{1\}$.  A calculation shows that, for $\lambda\in\rho(T)$ and $x\in X$, the resolvent satisfies $R(\lambda,T)x=y$, where, for each $k\geq1,$
 $$y_k=\frac{1}{\lambda^{1/2}}\sum_{n=0}^\infty(-1)^{n+1}\left(\frac{1}{(1-2\lambda^{1/2})^{n+1}}-\frac{1}{(1+2\lambda^{1/2})^{n+1}} \right)x_{k+n},$$   the complex plane being cut along the negative real axis. Thus, for $p\in\{1,\infty\}$, $\|R(\lambda,T)\|=O((|1-2\lambda^{1/2}|-1)^{-1})$ as $\lambda\to1$ through $\rho(T)$ and, by the Riesz-Thorin theorem, the same statement holds for   $p\in(1,\infty)$.  It follows, in particular, that $\|R(\e^{\ii\theta},T)\|=O(|\theta|^{-2})$ as $\theta\to0$. Since $\| R(\e^{\ii\theta},T)\|\geq c|\theta|^{-2}$ for all $\theta\in(-\pi,\pi]$ by \eqref{resolvent_lb} and the geometry of $\sigma(T)$, it follows from Corollary~\ref{cor:poly_bounds} that \eqref{poly_bounds} holds with $\alpha=2$ for some constants $c,C>0$ and  all sufficiently large $n\geq0$.  However, an explicit calculation involving Stirling's formula shows that, for $p\in\{1,\infty\}$,   the actual rate of decay satisfies $\|T^n(I-T)\|\sim2(\pi n)^{-1/2}$ as $n\to\infty$ and hence,  by another application of the Riesz-Thorin theorem, $\|T^n(I-T)\|=\Theta(n^{-1/2})$ as $n\to\infty$ also for $p\in(1,\infty)$. Thus  the logarithmic factor in \eqref{poly_bounds} is redundant in this case.
\end{ex}

Theorem~\ref{hilbert_poly_decay}  will show that, if the underlying space is a Hilbert space, then the logarithmic factor in \eqref{poly_bounds} can in fact  be dropped for any operator whose resolvent grows at most polynomially.  For general Banach spaces, however, this is not the case, as Theorem~\ref{log_optimality} below establishes. The proof of this result requires two lemmas. The first is a variant of \cite[Lemma~4.6.6]{Green_book}, which itself is a special form of Levinson's log-log theorem; see for instance \cite[VII~D7]{Koo}. Here, given a set $\Omega\subset\CC$, $\partial\Omega$ denotes the boundary of $\Omega$.

\begin{mylem}\label{Levinson}
Let $X$ be a complex Banach space, let $\theta\in(-\pi,\pi]$ and let $\Omega$ be a neighbourhood of the point $\e ^{\ii \theta}\in\T$.  Furthermore, given $r\in(0,1)$, let \begin{equation}\label{Omega_r}\Omega_{r,\theta}:=\left\{\lambda\in\CC:\left|\lambda-\e^{\ii\theta}\frac{1+r^2}{1-r^2}\right|\leq \frac{2r}{1-r^2}\right\}.\end{equation} Then there exists a constant $C>0$ with the following property: If $r\in(0,\frac{1}{4})$ is such that $\Omega_{2r,\theta}\subset \Omega$ and if $F:\Omega\to X$ is a holomorphic function such that, for some constant $B>0$, $\|F(\lambda)\|\leq B|1-|\lambda||^{-1}$ for all  $\lambda\in\partial\Omega_{2r,\theta}\backslash\T$, then  $\|F(\lambda)\|\leq BCr^{-1}$ for all $\lambda\in\Omega_{r,\theta}$.
\end{mylem}

\begin{proof}[\textsc{Proof}] 
Assume, without loss of generality, that $\theta=0$ and, as in the proof of Theorem~\ref{thm}, let $\varphi$ denote the M\"obius transformation defined by $\varphi(\lambda):=\frac{1-\lambda}{1+\lambda}$, so that $\varphi$  maps the circle $\gamma_r:=\partial\Omega_{r,0}$ onto $r \T$ for each $r\in(0,1)$. Moreover, by \eqref{Re} with $r$ replaced by $2r$, there exists a constant $C'>0$ which is independent of $r\in(0,\frac{1}{4})$ and such that   $|\R\varphi(\lambda)|\leq C'||\lambda|-1|$ for all $\lambda\in \gamma_{2r}$. Consider the function $G:\Omega\to X$ defined by $$G(\lambda):=\left(1+\frac{\varphi(\lambda)^2}{4r^2}\right)F(\lambda).$$ For $\lambda\in \gamma_{2r}$, the term in brackets has modulus $r^{-1}|\R\varphi(\lambda)|$ and hence, by the assumption on $F$,  $\|G(\lambda)\|\leq BC' r^{-1}$ for all such $\lambda$. Since $\Omega_{r,0}\subset\Omega_{2r,0}$, it follows from the maximum principle that  $\|G(\lambda)\|\leq BC' r^{-1}$ for all $\lambda\in \Omega_{r,0}$. But if $\lambda\in\Omega_{r,0}$, then $|\varphi(\lambda)|\leq r$ and hence $\|G(\lambda)\|\geq \frac{3}{4}\|F(\lambda)\|$, which gives the result  with $C=\frac{4}{3}C'$.
\end{proof}

The second auxiliary result is a technical one and analogous to \cite[Lemma~3.9]{BT}. Given $\alpha\geq1$ and $\lambda\in\CC\backslash\{0\}$,  let   \begin{equation}\label{K}K_\alpha(\lambda):=\frac{|\arg\lambda|^\alpha}{2\pi^\alpha},\end{equation} where the argument of a complex number is taken to lie in $(-\pi,\pi]$, and define the regions $\Omega_\alpha, \Theta_\alpha\subset\CC$ by  \begin{equation}\label{regions}\begin{aligned}\Omega_{\alpha}&:=\big\{\lambda\in \CC\backslash\{0\}: |\lambda| \leq 1-K_\alpha(\lambda)\big\}\cup\{0\}, \\\Theta_\alpha&:=\big\{\lambda\in \CC\backslash\{0\}:1-K_\alpha(\lambda)<|\lambda|<2\big\},\end{aligned}\end{equation} respectively, so that $\Theta_\alpha=2\DD\backslash \Omega_{\alpha}.$  Furthermore, given a complex measure $\mu$ whose support is contained in  $\Omega_\alpha$, define the transforms $\C_\alpha\mu$, $\LL_\alpha\mu$ and $\D_\alpha\mu$, for $\lambda\in\Theta_\alpha$, $k\geq1$ and $n\geq0$, respectively, by  \begin{equation}\label{transforms}\begin{aligned}(\C_\alpha\mu)(\lambda)&:=\int_{\Omega_\alpha}\frac{\dd\mu(z)}{\lambda-z}, \\(\LL_\alpha\mu)(k)&:=\int_{\Omega_\alpha}z^{k-1}\,\dd\mu(z),\\ (\D_\alpha\mu)(n)&:=\int_{\Omega_\alpha}z^{n}(1-z)\,\dd\mu(z).\end{aligned}\end{equation}  

\begin{mylem}\label{measure_construction}
Suppose that $\alpha>2$ and let the function $K_\alpha$, the regions $\Omega_\alpha$ and $\Theta_\alpha$,  and the transforms $\C_\alpha$, $\LL_\alpha$ and $\D_\alpha$  be defined as in \eqref{K}, \eqref{regions} and \eqref{transforms}, respectively. Then there exists a constant $C>0$ with the following property: Given any $n_0\in\NN$, there exists a complex measure $\mu$ whose support is contained in $\Omega_{\alpha}$ and  which is such that  \begin{enumerate}
\item[(i)] $\displaystyle K_\alpha(\lambda) |(\C_\alpha\mu)(\lambda)|\leq C$ for all $\lambda\in\Theta_\alpha$;
\item[(ii)] $\displaystyle  |(\LL_\alpha\mu)(k)|\leq C$ for all  $k\geq1$;
\item[(iii)] $|(\D_\alpha\mu)(n_1)|^\alpha\geq (C n_1)^{-1}\log n_1$ for some $n_1>n_0$.
\end{enumerate}
\end{mylem}

\begin{proof}[\textsc{Proof}] 
Choose $\theta\in(0,\frac{1}{2})$  and  $\beta\in(\frac{\alpha}{32},\frac{\alpha}{16})$ in such a way that  $\ell:=-\beta\theta^{-\alpha}\log\theta$ is an integer satisfying $\ell>\frac{n_0}{2}+2$ and that $\theta^{-(\alpha-2)}>2\alpha\beta^{-1}+1$. Now, with $B_\ell:=2 \ell\log_2 \ell$, $\zeta_\ell:=\e^{2\pi\ii/\ell}$ and $\lambda_0:=\frac{1}{2}\e ^{\ii \theta}$,  define the measure $\mu$ as $$\mu:=\frac{B_\ell^{\ell-1}}{ \ell^{1/2}}\sum_{r=0}^{\ell-1}\zeta_\ell^r\left(1+\frac{\zeta_\ell^r}{2B_\ell\lambda_0}\right)\delta_{\lambda_0+\frac{\zeta_\ell^r}{2B_\ell}},$$where $\delta_\lambda$ denotes the Dirac measure concentrated at  $\lambda$. 

Then, for any $\lambda\in\Theta_\alpha$,  \begin{equation*}\label{C_sum}(\C_\alpha\mu)(\lambda)=\frac{B_\ell^{\ell-1}}{ \ell^{1/2}}\sum_{r=0}^{\ell-1}\left(\frac{2B_\ell\zeta_\ell^r}{2B_\ell(\lambda-\lambda_0)-\zeta_\ell^r}+\frac{1}{\lambda_0}\frac{\zeta_\ell^{2r}}{2B_\ell(\lambda-\lambda_0)-\zeta_\ell^r}\right).\end{equation*} However, for  $1\leq j\leq \ell$ and  $\lambda\in\CC$ such that $\lambda^\ell\ne1$,  $$\sum_{r=0}^{\ell-1}\frac{\zeta_\ell^{jr}}{\lambda-\zeta_\ell^r}=\frac{\ell\lambda^{j-1}}{\lambda^\ell-1}$$ (see also the proof of \cite[Lemma~3.9]{BT}), and applying this with $j=1,2$ gives  \begin{equation}\label{C_eq}(\C_\alpha\mu)(\lambda)=\frac{\lambda}{\lambda_0}\frac{2 \ell^{1/2}}{2^\ell(\lambda-\lambda_0)^\ell-B_\ell^{-\ell}}\end{equation}
for all $\lambda\in\Theta_\alpha$. Since $B_\ell^{-1}\leq |\lambda-\lambda_0|$  and $|\lambda|\leq 2$ for all  $\lambda\in\Theta_\alpha$,  this in turn becomes  \begin{equation}\label{C_bound}|(\C_\alpha\mu)(\lambda)|\leq \frac{C\ell^{1/2}}{2^\ell|\lambda-\lambda_0|^\ell}.\end{equation}  Let $\lambda\in\Theta_\alpha$ be given. If $K_\alpha(\lambda)>\theta^\alpha$, then $|\arg\lambda|>\pi\theta$ and an elementary geometric argument shows that  $$|\lambda-\lambda_0|\geq\frac{1}{2}+\frac{1}{2}\big(1-\cos((\pi-1)\theta)-K_\alpha(\e^{\ii\pi\theta})\big)\geq\frac{1}{2}(1+\theta^2-\theta^\alpha).$$
Using the fact that $1+\theta^2-\theta^\alpha\geq\e^{\frac{1}{2}(\theta^2-\theta^\alpha)}$ for all $\theta\in(0,\frac{1}{2})$, it follows from \eqref{C_bound} that    $$|(\C_\alpha\mu)(\lambda)|\leq C\ell^{1/2}\e^{-\frac{\ell}{2}(\theta^2-\theta^\alpha)}= C(-\beta\log\theta)^{1/2}\theta^{-\frac{\alpha}{2}+\frac{\beta}{2}(\theta^{-(\alpha-2)}-1)}.$$ Now the choices of $\theta$ and $\beta$ ensure that the exponent of $\theta$ on the right-hand side of this expression is strictly greater than $\frac{\alpha}{2}$, and hence $|(\C_\alpha\mu)(\lambda)|$ and consequently $K_\alpha(\lambda)|(\C_\alpha\mu)(\lambda)|$ are uniformly bounded, independently of $\theta$ and $\beta$, for all  $\lambda\in\Theta_\alpha$ satisfying $K_\alpha(\lambda)>\theta^\alpha$. If $K_\alpha(\lambda)\leq\theta^\alpha$, on the other hand, then $|\lambda-\lambda_0|\geq\frac{1}{2}(1-2\theta^\alpha)$ and, using the fact that $1-2\theta^\alpha\geq\e^{-4\theta^\alpha}$ for all $\theta\in(0,\frac{1}{2})$,  \eqref{C_bound} gives 
$$|(\C_\alpha\mu)(\lambda)|\leq C \ell^{1/2}\e^{4\ell\theta^\alpha}=C(-\beta\log\theta)^{1/2}\theta^{-(\frac{\alpha}{2}+4\beta)}.$$  Since  the choice of $\beta$  ensures that  $\frac{\alpha}{2}+4\beta<\frac{3\alpha}{4}$,  $K_\alpha(\lambda)|(\C_\alpha\mu)(\lambda)|$ is uniformly bounded, again independently of $\theta$ and $\beta$, also for all  $\lambda\in\Theta_\alpha$ with $K_\alpha(\lambda)\leq\theta^\alpha$. This establishes (i) for $C=C_1$, where $C_1>0$ is some suitably large constant.

Next observe that, for each $k\geq1$,
\begin{equation}\label{Lmu}(\LL_\alpha\mu)(k)=\frac{B_\ell^{\ell-1}\lambda_0^{k-1}}{ \ell^{1/2}}\sum_{r=0}^{\ell-1}\zeta_\ell^r\left(1+\frac{\zeta_\ell^r}{2B_\ell\lambda_0}\right)^{k}.\end{equation}
Expanding and using the fact that, for any integer $s\geq0$, $$\sum_{r=0}^{\ell-1}\zeta_\ell^{r(s+1)}=\begin{cases}\ell &\text{if $s+1=0\!\!\!\pmod{\ell},$}\\ 0 &\text{otherwise,} \end{cases}$$this becomes 
\begin{equation}\label{L_sums}\begin{aligned}(\LL_\alpha\mu)(k)&=\frac{B_\ell^{\ell-1}\lambda_0^{k-1}}{ \ell^{1/2}}\sum_{s=0}^{k}\sum_{r=0}^{\ell-1}\binom{k}{s}\frac{\omega_\ell^{r(s+1)}}{(2B_\ell\lambda_0)^s}\\&= \ell^{1/2}B_\ell^{\ell-1}\lambda_0^{k-1}\sum_{r=1}^{\lfloor\frac{k+1}{\ell}\rfloor}\frac{\binom{k}{r \ell-1}}{(2B_\ell\lambda_0)^{r \ell-1}}.
\end{aligned}\end{equation}
Next note that, for $1\leq r\leq\lfloor\frac{k+1}{\ell}\rfloor$,  $\binom{k}{r \ell-1}\leq\binom{k}{\ell-1}\frac{(\ell-1)!}{(r\ell-1)!}k^{(r-1)\ell}$. Thus, for $1\leq k< B_\ell$,  $$|(\LL_\alpha\mu)(k)|\leq \frac{\ell^{1/2}}{2^{k-1}}\binom{k}{\ell-1}\sum_{r=1}^{\lfloor\frac{k+1}{\ell}\rfloor}\frac{(\ell-1)!}{(r\ell-1)!}\leq C\frac{\ell^{1/2}}{2^k}\binom{k}{\ell-1},$$ where $C$ is independent of $\ell$. Now, if $k\leq 2\ell-3$, then $\binom{k}{\ell-1}\leq\frac{1}{2}\binom{k+1}{\ell-1}$ and, if $k\geq2\ell-2$, then $\binom{k}{\ell-1}\leq2\binom{k-1}{\ell-1}$, so in either case $\frac{1}{2^k}\binom{k}{\ell-1}\leq \frac{1}{2^{2\ell-2}}\binom{2\ell-2}{\ell-1}$. Hence
 \begin{equation}\label{neg_bin}|(\LL_\alpha\mu)(k)|\leq C\frac{ \ell^{1/2}}{4^{\ell-1}}\binom{2\ell-2}{\ell-1},\end{equation} which by Stirling's formula is bounded above independently of $\ell$. If $k\geq B_\ell$, on the other hand, then \eqref{Lmu} and the fact that $1+B_\ell^{-1}\leq\e^{1/B_\ell}$ give
 $$|(\LL_\alpha\mu)(k)|\leq  \frac{\ell^{1/2}B_\ell^{\ell-1}}{2^{k-1}}\left(1+\frac{1}{B_\ell}\right)^k\leq 2\ell^{1/2}B_\ell^{\ell-1}\e^{-k(\log 2-B_\ell^{-1})}\leq C\frac{ \ell^{1/2}B_\ell^{\ell-1}}{2^{B_\ell}}$$ and, by the definition of $B_\ell$, the right-hand side is again bounded above independently of $\ell$. Thus (ii) holds with $C=C_2$ for some sufficiently large $C_2>0$.

In order to establish (iii), note first that,  for each  $n\geq0$, $$(\D_\alpha\mu)(n)=\frac{B_\ell^{\ell-1}\lambda_0^{n}}{ \ell^{1/2}}\sum_{r=0}^{\ell-1}\zeta_\ell^r\left(1+\frac{\zeta_\ell^r}{2B_\ell\lambda_0}\right)^{n+1}\left(1-\lambda_0\left(1+\frac{\zeta_\ell^r}{2B_\ell\lambda_0}\right)\right).$$ Thus, if $n\geq0$ is such that  $\lfloor\frac{n+2}{\ell}\rfloor=\lfloor\frac{n+3}{\ell}\rfloor$, proceeding as in \eqref{L_sums} gives
$$(\D_\alpha\mu)(n)= \ell^{1/2}B_\ell^{\ell-1}\lambda_0^{n}\sum_{r=1}^{\lfloor\frac{n+2}{\ell}\rfloor}\frac{\binom{n+1}{r \ell-1}-\lambda_0\binom{n+2}{r \ell-1}}{(2B_\ell\lambda_0)^{r \ell-1}}.$$ Now let $n_1:=2\ell-4$, so that $n_1>n_0$ and  $\lfloor\frac{n_1+2}{\ell}\rfloor=\lfloor\frac{n_1+3}{\ell}\rfloor=1$. Then $$|(\D_\alpha\mu)(n_1)|=\frac{ \ell^{1/2}}{4^{\ell-2}}\binom{2\ell-2}{\ell-1}\left|\frac{1}{2}-\lambda_0\right|,$$ and hence,  by another application of Stirling's formula, $|(\D_\alpha\mu)(n_1)|\geq c\theta$, where $c$ is independent $\theta$. Since the definition of $\ell$ implies that $\theta^\alpha\geq c n_1^{-1}\log n_1$, it follows that (iii) holds for  $C=C_3$, where $C_3>0$ is another suitably large constant.  Setting $C=\max\{C_1,C_2,C_3\}$ now completes the proof.
\end{proof} 

\begin{rem}\label{prob_rem}
The estimates leading to \eqref{neg_bin} may also be viewed in another way. Indeed, given $\ell\geq1$, let $Y_\ell$ be the random variable counting the number of tosses of a fair coin  required in order to obtain  a total of exactly $\ell$ heads, so that $Y_\ell$ has the negative binomial distribution with $\PP(Y_\ell=k)=\frac{1}{2^k}\binom{k-1}{\ell-1}$ for each $k\geq 1$. Then the aforementioned estimates amount to the observation that $Y_\ell$ has mode $2\ell-1$. This probabilistic interpretation will reappear in Remark~\ref{yuri_rem} below.
\end{rem}

The following result, which is an analogue of \cite[Theorem~4.1]{BT}, shows that the logarithmic factor in Corollary~\ref{cor:poly_bounds} cannot in general be omitted.

\begin{mythm}\label{log_optimality}
Given any $\alpha>2$, there exists a non-trivial complex Banach space $X_\alpha$ and a power-bounded operator $T\in\B(X_\alpha)$ such that  $\sigma(T)\cap\T=\{1\}$ and $\|R(\e ^{\ii \theta},T)\|=O(|\theta|^{-\alpha})$ as $\theta\to0$, and for which  \begin{equation}\label{limsup}\limsup_{n\to\infty}\|T^n(I-T)\|\left(\frac{n}{\log n}\right)^{1/\alpha}>0.\end{equation}
\end{mythm}

\begin{proof}[\textsc{Proof}] 
Given any sequence $x\in\ell^\infty$,  define the function $F_x$,  for $|\lambda|>1$, by $$F_x(\lambda):=\sum_{k=1}^\infty \frac{x_k}{\lambda^{k}}.$$ Now, with  $K_\alpha$, $\Omega_\alpha$ and $\Theta_\alpha$  as defined in \eqref{K} and \eqref{regions}, let $X_\alpha$ denote the subspace  of $\ell^\infty$ consisting of sequences $x$ for which $F_x$ extends analytically to $\Theta_\alpha$ and satisfies $\sup\{K_\alpha(\lambda)|F_x(\lambda)|:\lambda\in\Theta_\alpha\}<\infty$. This space is non-trivial, containing for instance any finitely supported sequence  as well as the constant sequence $(1,1,1,\dots)$, and, by an application of Vitali's theorem, it is complete  under the norm $\|\cdot\|_{X_\alpha}$ given by $\|x\|_{X_\alpha}:=\|x\|_\infty+\| x\|_\alpha,$ where $ \|x\|_\alpha:=\sup\{K_\alpha(\lambda)|F_x(\lambda)|:\lambda\in\Theta_\alpha\}.$   Consider the restriction $T:=S|_{X_\alpha}$ to $X_\alpha$ of the left-shift operator $S\in\B(\ell^\infty)$. 
 
 Given    $x\in X_\alpha$ and $|\lambda|>1$,  $F_{Tx}(\lambda)=\lambda F_x(\lambda)-x_1$, so $F_{Tx}$ extends analytically to $\Theta_\alpha$ and $\|Tx\|_\alpha\leq \frac{1}{2}\|x\|_\infty+2\|x\|_\alpha$. Thus $T$ maps $X_\alpha$ into itself and defines an element of $\B(X_\alpha)$ with  norm $\|T\|\leq2$. More generally, having fixed some $x\in X_\alpha$ and given $n\geq0$, let   $F_n:=F_{T^nx}$. Then $F_{n}$ extends analytically to $\Theta_\alpha$ and is given, for $|\lambda|>1$,  by
  \begin{equation}\label{Laplace_of_shift}F_n(\lambda)=\sum_{k=1}^\infty\frac{x_{n+k}}{\lambda^{k}}=\lambda^nF_{x}(\lambda)-\sum_{k=1}^{n}\lambda^{n-k}x_k.\end{equation}  
 Writing $\A$ for the annulus $\{\lambda\in\CC:1<|\lambda|<2\}$, it follows that $$|F_n(\lambda)|\leq\begin{cases}\frac{\|x\|_\infty}{|1-|\lambda||}&\text{if $\lambda\in\A$,} \\\frac{\|x\|_\infty}{|1-|\lambda||}+|F_{x}(\lambda)|&\text{if $\lambda\in\Theta_\alpha\cap\DD$,}\end{cases}$$
and, in particular, $|F_n(\lambda)|\leq |1-|\lambda||^{-1}\|x\|_{X_\alpha}$ for all $\lambda\in\Theta_\alpha\backslash\T$. Let $\lambda\in\Theta_\alpha$ be given. If $\lambda\in(1,2)$, then  $K_\alpha(\lambda)|F_n(\lambda)|=0$. Suppose therefore that  $\theta:=\arg\lambda$ satisfies $0<|\theta|\leq\pi$, and note that  $\Omega_{2r_\lambda,\theta}\subset\Theta_\alpha$, where $r_\lambda:=\frac{1}{4}K_\alpha(\lambda)$ and  $\Omega_{r,\theta}$ is  defined, for $r\in(0,1)$, as in \eqref{Omega_r}. Now either $|1-|\lambda||>r_\lambda$, in which case $K_\alpha(\lambda)|F_n(\lambda)|\leq C\|x\|_\alpha$ for some constant $C$ which is independent of  $x\in X_\alpha$, $n\geq0$ and $\lambda$, or $|1-|\lambda||\leq r_\lambda$. In the latter case $\lambda\in\Omega_{r_\lambda,\theta}$, so the same estimate follows from Lemma~\ref{Levinson} applied to the function $F_n$ on the disc $\Omega_{2r_\lambda,\theta}$. Thus $\|T^nx\|_\alpha \leq C\|x\|_{X_\alpha}$ for some constant $C$ which is independent of $x\in X_\alpha$ and  $n\geq0$. Since moreover $\|T^n x\|_\infty\leq\|x\|_\infty$ for all $n\geq0$, it follows that $T$ is power-bounded.
  
Now fix $x\in X_\alpha$ and let $\Omega:=\{\lambda\in\CC:1<|\lambda|<\frac{3}{2}\}$. Then $\Omega\subset\rho(T)$ and  $(R(\lambda,T)x)_{n+1}=F_{n}(\lambda)$ for all $\lambda\in\Omega$ and all $n\geq0$, where $F_n$ is as above. In particular, \eqref{Laplace_of_shift}  remains true for each $n\geq0$ when the left-hand side is replaced by $(R(\lambda,T)x)_{n+1}$, so the argument in the previous paragraph shows that   \begin{equation}\label{inf_bound}K_\alpha(\lambda)\|R(\lambda,T)x\|_\infty\leq  C\|x\|_{X_\alpha},\end{equation} where $C$ is independent of both  $x\in X_\alpha$ and $\lambda\in\Omega$. The aim now is to show that  $K_\alpha(\lambda)\|R(\lambda,T)x\|_\alpha\leq C\|x\|_\alpha$ for all $\lambda\in \Omega$, from which it will follow that the norm of the resolvent of $T$ grows at most polynomially. Since the estimate holds trivially when $\lambda$ is real, assume that $\lambda\in\Omega$ satisfies $0<|\arg\lambda|\leq\pi$ and let $F_\lambda:=F_{R(\lambda,T)x}$. Then, for $|\mu|\geq\frac{3}{2}$, 
$$\begin{aligned}
F_\lambda(\mu)=\sum_{n=1}^\infty\sum_{k=1}^\infty\frac{x_{k+n-1}}{\lambda^{k}\mu^{n}}=\sum_{n=1}^\infty\frac{\lambda^{n-1}}{\mu^{n}}\left(F_x(\lambda)-\sum_{k=1}^{n-1}\frac{x_{k}}{\lambda^{k}}\right)=-\frac{F_{x}(\lambda)-F_x(\mu)}{\lambda-\mu},
\end{aligned}$$
so $F_\lambda$ extends analytically to $\Theta_\alpha$, taking the values
 \begin{equation}\label{F_lambda}
 F_\lambda(\mu)=\begin{cases}-\frac{F_{x}(\lambda)-F_{x}(\mu)}{\lambda-\mu}&\text{if $\mu\in\Theta_\alpha\backslash\{\lambda\}$,} \\-F_x'(\lambda)&\text{if $\mu=\lambda$.}\end{cases} \end{equation}
Now let $\mu\in\Theta_\alpha$  and set $M_\alpha(\lambda,\mu):=\frac{1}{4}\max\{K_\alpha(\lambda),K_\alpha(\mu)\},$ which is positive by the assumption on $\arg\lambda$. If $|\lambda-\mu|>M_\alpha(\lambda,\mu)$, then by \eqref{F_lambda}  
\begin{equation}\label{est1} K_\alpha(\lambda)K_\alpha(\mu)\left|F_\lambda(\mu)\right|\leq\frac{K_\alpha(\lambda)+K_\alpha(\mu)}{M_\alpha(\lambda,\mu)}\|x\|_\alpha\leq C\|x\|_\alpha,\end{equation}
 where $C$ is independent of $\lambda$ and $\mu$. Now suppose that $|\lambda-\mu|\leq M_\alpha(\lambda,\mu)$. By Cauchy's formula, \begin{equation}\label{Cauchy_int}F_\lambda(\mu)=\frac{1}{2\pi\ii }\oint_{\Gamma}\frac{F_x(\lambda)-F_x(z)}{(\lambda-z)(z-\mu)}\,\dd z,\end{equation} where $\Gamma$ is any contour in $\Theta_\alpha$  whose interior  contains the point $\mu$ and is itself contained in $\Theta_\alpha$. If $|\lambda-\mu|\leq \frac{1}{4}K_\alpha(\lambda),$ choose $\Gamma$ to be the circle with centre $\lambda$ and  radius $\frac{1}{2}K_\alpha(\lambda)$, so that $\Gamma\subset\Theta_\alpha$ by the definitions of $\Theta_\alpha$ and $\Omega$. Elementary estimates show that, for any $z\in\Theta_\alpha$ satisfying $|\lambda-z|\leq\frac{1}{2}K_\alpha(\lambda)$,  $c|\arg\lambda|\leq|\arg z|\leq C|\arg\lambda|$ and hence $cK_\alpha(\lambda)\leq K_\alpha(z)\leq CK_\alpha(\lambda)$, where $c$ and $C$ are independent of $\lambda$ and $z$.  This applies in particular to all $z\in\Gamma$ and also to $z=\mu$. Since moreover  $|\mu-z|\geq \frac{1}{4}K_\alpha(\lambda)$  for all $z\in \Gamma$, it follows from  \eqref{Cauchy_int}  that 
 \begin{equation}\label{est2}K_\alpha(\lambda)K_\alpha(\mu)|F_\lambda(\mu)|\leq \frac{C\|x\|_\alpha}{K_\alpha(\lambda)}\oint_\Gamma\left(\frac{K_\alpha(\mu)}{K_\alpha(\lambda)}+\frac{K_\alpha(\mu)}{K_\alpha(z)}\right)|\dd z|\leq C\|x\|_\alpha, \end{equation}
 where $C$ depends neither on $\lambda$ nor on $\mu$.  A similar argument applies when $\frac{1}{4}K_\alpha(\lambda)<|\lambda-\mu|\leq \frac{1}{4}K_\alpha(\mu)$, this time taking $\Gamma$ to be the circle with centre $\lambda$ and radius $\frac{1}{2}K_\alpha(\mu)$. Then  $|\mu-z|\leq\frac{3}{4}K_\alpha(\mu)$ for all $z\in\Gamma$, so  $\Gamma\subset\Theta_\alpha$ as before.  Moreover,   $K_\alpha(\mu)\leq CK_\alpha(z)$ for all $z\in\Gamma$, where $C$ is independent of $\lambda$ and $\mu$, and $K_\alpha(\lambda)<K_\alpha(\mu)$, giving 
 \begin{equation}\label{est3}K_\alpha(\lambda)K_\alpha(\mu)|F_\lambda(\mu)|\leq \frac{C\|x\|_\alpha}{K_\alpha(\mu)}\oint_\Gamma\left(1+\frac{K_\alpha(\lambda)}{K_\alpha(z)}\right)|\dd z|\leq C\|x\|_\alpha. \end{equation}
Combining \eqref{est1}, \eqref{est2} and \eqref{est3}  shows that $K_\alpha(\lambda)\|R(\lambda,T)x\|_\alpha\leq  C\|x\|_\alpha$ for all $\lambda\in\Omega$. Together   with \eqref{inf_bound}, this gives  $K_\alpha(\lambda)\|R(\lambda,T)x\|_{X_\alpha}\leq  C\|x\|_{X_\alpha}$, where $C$ is independent of $x\in X_\alpha$ and $\lambda\in\Omega$, and hence $\sup\{K_\alpha(\lambda)\|R(\lambda,T)\|:\lambda\in\Omega\}<\infty$. In particular,  it follows from by  \eqref{resolvent_lb} that $\sigma(T)\cap\T\subset\{1\}$,  and a simple approximation argument shows that $\|R(\e ^{\ii \theta},T)\|=O(|\theta|^{-\alpha})$ as $\theta\to0$. Furthermore, since  $(1,1,1,\dots)$ is a fixed point of $T$, $1\in\sigma(T)$.

Finally, let the transforms $\C_\alpha$, $\LL_\alpha$ and $\D_\alpha$  be as defined in \eqref{transforms} and note that, given any complex measure $\mu$ whose support is contained  in $\Omega_\alpha$ and for which $\sup\{|(\LL_\alpha\mu)(k)|:k\geq1\}<\infty$, there exists an associated sequence $x^{\mu}\in \ell^\infty$ whose entries are given, for each $k\geq1$, by $x^{\mu}_k:=(\LL_\alpha\mu)(k)$. By Fubini's theorem, $$F_{x^{\mu}}(\lambda)=\sum_{k=1}^\infty\int_{\Omega_\alpha}\frac{z^{k-1}}{\lambda^{k}}\,\dd\mu(z)=\int_{\Omega_\alpha}\frac{\dd\mu(z)}{\lambda-z}=(\C_\alpha\mu)(\lambda)$$ whenever  $|\lambda|>1$,   so $x^{\mu}\in X_\alpha$ provided $\sup\{K_\alpha(\lambda)|(\C_\alpha\mu)(\lambda)|:\lambda\in\Theta_\alpha\}<\infty$.  Note also that, for each $n\geq0$, $(\D_\alpha\mu)(n)=x_{n+1}^{\mu}-x_{n+2}^{\mu}$, which coincides with the first entry of $T^n(I-T)x^{\mu}$. Now, by Lemma~\ref{measure_construction},  it is possible to find a sequence $(n_j)$ of integers, with  $n_j\to\infty$  as $j\to\infty$,  and associated measures $\mu_j$ such that  $\{x^{\mu_j}: j\geq1\}$ is a  bounded subset of $X_\alpha$
and moreover  $|(\D_\alpha\mu_j)(n_j)|^\alpha\geq cn_j^{-1}\log n_j$ for each $j\geq1$. By rescaling if necessary, there is no loss of generality in assuming that $\|x^{\mu_j}\|_{X_\alpha}\leq1$ for all $j\geq1$, so that  $$\|T^{n_j}(I-T)\|\geq \|T^{n_j}(I-T)x^{\mu_j}\|_\infty\geq|(\D_\alpha\mu_j)(n_j)|.$$ Hence \eqref{limsup} holds and the proof is complete.
\end{proof}

\begin{rem}
\label{yuri_rem}
It is possible to replace Lemma~\ref{measure_construction}, which here gives rise to the sequences $x^{\mu}\in X_\alpha$ used to establish to \eqref{limsup}, by a simpler, more ad-hoc construction. Indeed, using the notation introduced in the proof of that result,  let  $$x_k:=\frac{\ell^{1/2}\lambda_0^{k-\ell}}{2^{\ell-1}}\binom{k}{\ell-1}$$ for each $k\geq1$, so that $x_k$ equals the first term of the final sum in \eqref{L_sums} which defines $x_k^\mu$ in the above proof. In the notation of Remark~\ref{prob_rem}, this becomes $$x_k=4\ell^{1/2}(2\lambda_0)^{k-\ell}\PP(Y_\ell=k+1),$$ so the formula for the probability generating function of $Y_\ell$ (see for instance \cite[Section~4.2]{GrWe}) gives $$F_x(\lambda)=\frac{\lambda}{\lambda_0}\frac{2\ell^{1/2}}{2^\ell(\lambda-\lambda_0)^\ell}$$ whenever  $|\lambda|>1$, which should be compared with the right-hand sides of \eqref{C_eq} and \eqref{C_bound}. Since the estimates for $x^\mu$ established in  Lemma~\ref{measure_construction} apply equally to $x$, it follows that  \eqref{limsup} may also be obtained using sequences of this simpler form in the final paragraph of the above proof.
\end{rem}

\begin{rem}
 It is unclear whether Theorem~\ref{log_optimality} can be extended, for instance by modifying the construction in Lemma~\ref{measure_construction}, to  the range $1<\alpha\leq2$.  Note however that, by  Theorem~\ref{Ritt_thm}, the case $\alpha=1$ is necessarily excluded. See \cite[Theorem~1.2]{Dun} for a result relating specifically to the case $\alpha=2$.
\end{rem}

Theorem~\ref{hilbert_poly_decay} below shows that the situation is different when $X$ is  a Hilbert space. It relies on the following preparatory result, which is analogous to \cite[Lemma~2.3]{BT} (see also \cite[Lemma~1.1]{BEPS} and \cite[Lemma~3.2]{LS}) and holds for general Banach spaces. Recall that, if $T\in\B(X)$ is a power-bounded operator with $M:=\sup\{\|T^n\|:n\geq0\}$ and if $\lambda\in\CC$ satisfies $\R\lambda<0$, then it follows from \eqref{eq:resolvent_bound} that $\|R(\lambda,I-T)\|\leq M|\R\lambda|^{-1}$, and hence that the operator $I-T$ is sectorial. Thus, given any $s>0$, the fractional power $(I-T)^s$ is defined as $$(I-T)^s:=\frac{1}{2\pi\ii}\oint_\Gamma\lambda^sR(\lambda,I-T)\,\dd \lambda,$$ where the complex plane is cut along the negative real axis and where $\Gamma$ is any suitable contour  that contains  the point $1$ and otherwise encloses $\sigma(T)$ without touching it. Fractional powers coincide with the usual ones whenever $s\in\NN$, and moreover $(I-T)^{s+t}=(I-T)^s(I-T)^t$ for all $s,t>0$; see for instance  \cite{HaTo} for details.

\begin{mylem}\label{resolvent_lem}
Let $X$ be a complex Banach space and let $T\in\B(X)$ be a power-bounded operator such that $\sigma(T)\cap\T=\{1\}$. Furthermore, let $\alpha\geq1$ and suppose that $\|R(\e^{\ii\theta},T)\|=O(|\theta|^{-\alpha})$ as $\theta\to0$. Then $\sup\{\|(I-T)^\alpha R(\lambda,T)\|:|\lambda|>1\}<\infty$.
\end{mylem}

\begin{proof}[\textsc{Proof}] 
By \eqref{eq:resolvent_bound} it suffices to prove that $\sup\{\|(I-T)^\alpha R(\lambda,T)\|:\lambda\in\A\}<\infty,$ where $\A:=\{\lambda\in\CC:1<|\lambda|<2\}$.  A first step towards this result is to establish that, under the above assumptions,   $\sup\{\|(1-\lambda)^\alpha R(\lambda,T)\|:\lambda\in\A\}<\infty$. Thus, given $r\in (0,1)$, let $$\Omega_r:=\left\{\lambda\in\CC:1\leq|\lambda|\leq2\;\mbox{and}\;\left|\frac{1+r^2}{1-r^2}-\lambda\right|\geq \frac{2r}{1-r^2}\right\}$$ and let the map $H_r:\Omega_r\to\B(X)$ be defined by $$H_r(\lambda):=\left(1+\frac{r^2}{\varphi(\lambda)^2}\right)(1-\lambda)^\alpha R(\lambda,T),$$ so that  $H_r(\lambda)=h_r(\lambda)(1-\lambda)^\alpha R(\lambda,T),$ where $h_r:=g_r\circ\varphi$ with $g_r$ and $\varphi$ as in  the proof of Theorem~\ref{thm}. Then $\sup\{|h_r(\lambda)|:r\in(0,\frac{1}{2}),\lambda\in\Omega_r\}<\infty$ and, by the argument leading to equation~\eqref{h_bound}, $|h_r(\lambda)|\leq Cr^{-1}(|\lambda|-1)$ for all $\lambda\in\partial\Omega_r\cap\gamma_r$,  where $\gamma_r$ is as defined in \eqref{gamma_r}. Note also that $|1-\lambda|\leq C r$ for all $\lambda\in\partial\Omega_r\cap\gamma_r$. Thus, by \eqref{eq:resolvent_bound} and the assumption on the resolvent, $\|H_r(\lambda)\|$ is uniformly bounded, independently of $r$, for all $\lambda\in\partial\Omega_r$ and hence, by the maximum principle, $\sup\{\|H_r(\lambda)\|:r\in(0,\frac{1}{2}),\lambda\in\Omega_r\}<\infty$.  Since, given  any $\lambda\in\A$, there exists $r\in(0,\frac{1}{2})$ such that $\lambda\in\Omega_r$ and $|h_r(\lambda)|\geq\frac{1}{2}$, the claim  follows.

Now let  $n\in\NN$ and  $\beta\in[0,1)$ be such  that $\alpha=n+\beta$, and note that, for any $k\geq0$ and $\lambda\in\rho(T)$,   \begin{equation*}\label{res_exp}\|(I-T)^k R(\lambda,T)\|\leq|1-\lambda|^k\|R(\lambda,T)\|+\sum_{j=0}^{k-1}\binom{k}{j}|1-\lambda|^j\|\lambda-T\|^{k-j-1}.\end{equation*}   
Setting  $k=n-1$ and $k=n$, this shows, respectively, that $$\|(I-T)^{n-1}R(\lambda,T)\|\leq \frac{C}{|1-\lambda|^{1+\beta}}\quad\mbox{and}\quad\|(I-T)^{n}R(\lambda,T)\|\leq \frac{C}{|1-\lambda|^{\beta}}$$ for all $\lambda\in \A$. In particular, if $\beta=0$, the proof is complete.  If $\beta\ne0$, on the other hand, the moment inequality (see for instance \cite[Corollary~7.2]{Ha}) gives $$\|(I-T)^{\alpha-1} R(\lambda,T)\|\leq C\|(I-T)^{n-1} R(\lambda,T)\|^{1-\beta} \|(I-T)^n R(\lambda,T)\|^\beta,$$ and hence  $\|(I-T)^{\alpha-1}R(\lambda,T)\|\leq C|1-\lambda|^{-1}$ for all $\lambda\in\A$. Since $$\|R(\lambda,T)(I-T)^{\alpha}\|\leq |1-\lambda|\|R(\lambda,T)(I-T)^{\alpha-1}\|+\|(I-T)^{\alpha-1}\|$$ for all $\lambda\in\rho(T)$, the result follows.
\end{proof}

The final result shows that the phenomenon described in Theorem~\ref{log_optimality} cannot arise on Hilbert space. For analogous results in the continuous-time setting see \cite[Theorem~2.4]{BT} and \cite[Theorem~7.6]{BCT}; compare also with \cite[Theorem~9]{Nev2}.

\begin{mythm}\label{hilbert_poly_decay}
Let $X$ be a complex Hilbert space and let $T\in\B(X)$ be a power-bounded operator such that $\sigma(T)\cap\T=\{1\}$.  Furthermore, let $\alpha\geq1$. Then $\|R(\e^{\ii\theta},T) \|=O(|\theta|^{-\alpha})$ as $\theta\to0$ if and only if $\|T^n(I-T)\|=O(n^{-1/\alpha})$ as $n\to\infty$.
\end{mythm}

\begin{proof}[\textsc{Proof}] 
Suppose  that $\|R(\e^{\ii\theta},T) \|=O(|\theta|^{-\alpha})$ as $\theta\to0$, so that, by Lemma~\ref{resolvent_lem},  $\sup\{\|(I-T)^\alpha R(\lambda,T)\|:|\lambda|>1\}<\infty$. For $n\geq0$ and $|\lambda|>1$, let $$F_n(\lambda):=\lambda R(\lambda,T)\sum_{k=0}^n\lambda^{-k}T^k.$$ Then a simple calculation using the series expansion for the resolvent shows that $$F_n(\lambda)=\sum_{k=0}^\infty(\min\{k,n\}+1) \lambda^{-k}T^k,$$ and hence, by Parseval's identity, $$\sum_{k=0}^\infty(\min\{k,n\}+1)^2\frac{\|T^kx\|^2}{r^{2k}}=\frac{1}{2\pi}\int_0^{2\pi}\big\|F_n\big(r\e^{\ii\theta}\big)x\big\|^2\,\dd\theta$$ for all $n\geq0$,  $x\in X$ and  $r>1$. Replacing $x$ with $(I-T)^\alpha x$ and letting $B:=\sup\{\|(I-T)^\alpha R(\lambda,T)\|:|\lambda|>1\}$, it follows from the definition of $F_n$ that $$\sum_{k=0}^n\frac{(k+1)^2}{r^{2k}}\|T^k(I-T)^\alpha x\|^2\leq\frac{B^2r^2}{2\pi}\int_0^{2\pi}\left\| \sum_{k=0}^nr^{-k}\e^{-\ii k\theta}T^kx\right\|^2\,\dd\theta,$$or indeed $$\sum_{k=0}^n\frac{(k+1)^2}{r^{2n}}\|T^k(I-T)^\alpha x\|^2\leq B^2r^2\sum_{k=0}^n\frac{\|T^kx\|^2}{r^{2k}},$$by another application of Parseval's identity.  Letting $r\to1+$, this gives \begin{equation}\label{sum_bd}\sum_{k=0}^n(k+1)^2\| T^k(I-T)^\alpha x\|^2\leq M^2B^2(n+1)\|x\|^2,\end{equation} where $M:=\sup\{\|T^n\|:n\geq0\}$. Now, for  $y\in X$ and $n\geq0$, $$\big((n+2)T^n(I-T)^\alpha x,y\big)=\frac{2}{n+1}\sum_{k=0}^n\big((k+1)T^k(I-T)^\alpha x,(T^*)^{n-k}y\big),$$ where $T^*$ denotes the adjoint of $T$. By \eqref{sum_bd} and Cauchy's inequality, the right-hand side is bounded above in modulus by $2M^2B\|x\|\|y\|$, and hence 
\begin{equation}\label{a=1}\|T^n(I-T)^\alpha\|\leq\frac{2M^2B}{n+2}\end{equation} 
for all  $n\geq0$. Thus the proof is complete in the case $\alpha=1$. If $\alpha>1$, on the other hand,  the moment inequality gives  $$\|T^n(I-T)\|\leq C\|T^n\|^{(\alpha-1)/\alpha}\|T^n(I-T)^\alpha\|^{1/\alpha}$$ for all $n\geq0$, and the result now follows from \eqref{a=1} and the fact that $T$ is power-bounded. 

The converse implication is a consequence of Theorem~\ref{thm2}.
\end{proof}

\begin{rem}\label{Q_rem}
The above proof follows the method used in \cite{BCT}. An alternative approach, analogous to that of \cite{BT}, is to consider the operator $Q\in\B(X\times X)$ given by $Q(x,y):=(Tx+ T(I-T)^\alpha y,Ty)$. Then, for $n\geq0$, $Q^n$ is represented by the matrix $$Q^n=
\left(
\begin{array}{ccc}
  T^n&nT(I-T)^\alpha     \\
  0&T^n       
\end{array}
\right)$$
and, in particular,  $Q$ is power-bounded if and only if $\sup\{\|nT^n(I-T)^\alpha\|:n\geq0\}<\infty$. Since the latter is equivalent, by the moment inequality, to having $\|T^n(I-T)\|=O(n^{-1/\alpha})$ as $n\to\infty$, the main implication of Theorem~\ref{hilbert_poly_decay} can  be deduced from  results in \cite{Gom}, which characterise power-boundedness of an operator on a Hilbert space in terms of a certain integrability condition on its resolvent.
\end{rem}

\begin{rem} As in \cite[Theorem~2.4]{BT}, the equivalent statements in Theorem~\ref{hilbert_poly_decay} are also equivalent to the condition that, for every $x\in X$, $\|T^n(I-T)x\|=o(n^{-1/\alpha})$ as $n\to \infty$, which in turn is equivalent, by another application of the moment inequality, to having $nT^n(I-T)^\alpha\to0$ in the strong operator topology as $n\to\infty$. One implication follows from the general observation that, given any complex Banach space $X$ and a power-bounded mean ergodic operator $T\in\B(X)$ satisfying $\sigma(T)\cap\T\subset\{1\}$, the powers $T^n$ converge strongly, as $n\to\infty$,  to the projection $P$ onto $\Fix(T)$ along the closure of ${\Ran}(I-T)$; see  \cite[Theorem~4.1]{BaLy}. Indeed, if $\|T^n(I-T)\|=O(n^{-1/\alpha})$ as $n\to\infty$, then the operator $Q\in\B(X\times X)$ defined in Remark~\ref{Q_rem} is power-bounded, and furthermore  $\sigma(Q)=\sigma(T)$ and $\Fix(Q)=\Fix(T)\times\Fix(T)$. Hence applying this observation to $Q$ shows that, for any $x,y\in X$, $T^nx+nT^n(I-T)^\alpha y\to Px$ as $n\to\infty$. Since $T^nx\to Px$ as $n\to\infty$ by the same observation applied to $T$, it follows that $nT^n(I-T)^\alpha\to0$ in the strong operator topology as $n\to\infty$. The converse implication is a simple consequence of the  Uniform Boundedness Theorem.
\end{rem}

\section*{Acknowledgements}

The author is grateful to Professor C.J.K.\ Batty for his guidance and his careful reading of an earlier version of this work, to Professor Y.\ Tomilov for sharing an observation that led to Remark~\ref{yuri_rem} and for a number of helpful discussions during his visit to Oxford in autumn 2012, and finally to the EPSRC for its financial support.

\end{document}